\documentclass{amsart}
\usepackage{amssymb}
\usepackage{amsthm,mathtools}

\usepackage{graphicx, color}
\usepackage{scalerel,lipsum,url}
\usepackage{geometry, longtable}

\chardef\bslash=`\\ % p. 424, TeXbook

\newtheorem{thm}{Theorem}[section]
\newtheorem{cor}[thm]{Corollary}
\newtheorem{lem}[thm]{Lemma}
\newtheorem{prop}[thm]{Proposition}

\theoremstyle{definition}
\newtheorem{defn}{Definition}[section]

\newtheorem*{notation}{Notation and Terminology}

\theoremstyle{remark}

\usepackage[draft,inline,nomargin,index]{fixme}

\newcommand{\eval}[2][\right]{\relax
  \ifx#1\right\relax \left.\fi#2#1\rvert}

\title{Signatures of alternating group actions with non-zero quotient genus}

\author{Jennifer Paulhus}
\address{Mount Holyoke College}
\email{jpaulhus@mtholyoke.edu}

\author{Aaron Wootton}
\address{The University of Portland}
\email{wootton@up.edu}

\date{\today}

\begin{document}

\begin{abstract} We classify up to signature all the ways the alternating group $A_n$ can act on a compact Riemann surfaces when the quotient genus is greater than $0$. In particular, we prove that for $A_n$ with $n>6$ every potential signature for the group acting with quotient genus greater than $0$ is an actual signature. We also show that in the cases $n=5$ and $n=6$, the only failures occur for $[1;2]$ and $[1;3]$, respectively.  Along the way we also prove that for any finite simple non-abelian group, all potential signatures with quotient genus greater than $1$ are actual signatures.
\end{abstract}
%\begin{keyword}automorphisms of Riemann surfaces, generating sets of alternating groups, simple group actions on Riemann surfaces\end{keyword}

\maketitle

\section{Introduction}

The study of automorphism groups of compact Riemann surfaces, or equivalently monodromy groups of transitive branched covers, is classical in nature, with initial results dating back to Klein \cite{Klein1878}. However, it has also experienced a renaissance in the last 30 years due to significant advances in Group Theory, such as the classification of finite simple groups, and the rise of computer algebra systems such as GAP or Magma which make once insurmountable problems solvable with just a few lines of code. Continued interest in this area also stems from its application to other areas of mathematics such as the study of Mapping Class Groups \cite{NAKAMURA20183585} and Inverse Galois Theory \cite{Fried1991}.

An early and long established result of Hurwitz \cite{Hurwitz1893} states that every finite group $G$ acts on some compact Riemann surface of genus greater than 1. Kulkarni \cite{KULKARNI1987195} further showed that for a given finite group $G$, there is an infinite sequence of genera of Riemann surfaces, called the {\it genus spectrum} on which $G$ acts on a surface of each genus in this sequence. A next natural step in the study of group actions, then, is to examine the different ways in which a group can act within its genus spectrum. 

There are many different ways to distinguish between group actions on a compact Riemann surface \cite{Broughton2022FutureDI} with many classification schemes fully understood for low genera and certain families of groups, see for example \cite{BROUGHTON1991233} and \cite{10.1093/qmath/17.1.86}. A first step in most classification schemes, and a rudimentary classification in its own right, is to determine with which signatures a given group $G$ can act. This will be the focus of our work here.

Throughout the paper, $G$ will represent a finite group, and 
$\mathcal{O}(G) =\{\text{ord}(g) : g  \in G\} -\{1\}$ the {\it order set} of $G$.  For $a,b \in G$, we use standard notation for their {\it commutator}: $[a,b]=a^{-1}b^{-1}ab$. A tuple $[h; n_1,\ldots, n_r]$ is the {\em signature} of a finite group $G$ acting on a compact Riemann surface $S$ of genus $\sigma \geq 2$ if the quotient surface $S/G$ has genus $h$ and the quotient map $\pi \colon S\rightarrow S/G$ is branched over $r$ points with orders $n_1,\ldots, n_r$. We call $h$ the {\em quotient genus} and  the numbers $n_1,\ldots, n_r$ {\em periods} of the tuple. When $r=0$, so there are no periods, we write $[h;-]$. Riemann's Existence Theorem provides arithmetic and group theoretic conditions which are necessary and sufficient for a tuple $[h; n_1,\ldots, n_r]$ to be the signature of $G$ acting on $S$:

\begin{thm}[Riemann's Existence Theorem]\label{T:RET}
 A finite group $G$ acts on a compact Riemann surface $S$ of genus $\sigma \geq 2$ with signature $[h; n_1,\ldots, n_r]$ if and only if:
\begin{enumerate}
\item The Riemann--Hurwitz formula is satisfied:  $$\sigma =1+|G|(h-1) +\frac{|G|}{2} \sum_{j=1}^{r} \bigg( 1-\frac{1}{n_{j}} \bigg), $$ and

\item there exists a vector $(a_1,b_1,\dots, a_h,b_h,c_1,\dots ,c_r)$ of length $2h+r$  of elements of $G$ called a $[h; n_1,\ldots, n_r]$-generating vector for $G$ which satisfies the following properties:

\begin{enumerate}
\item $G=\langle a_{1},b_{1},a_{2},b_{2},\dots ,a_h,b_h,c_{1},\dots ,c_{r} \rangle$.
\item The order of $c_{j}$ is $n_{j}$ for $1\leq j\leq r$.
\item $\prod_{i=1}^{h} [a_{i} ,b_{i} ] \prod_{j=1}^{r} c_{j} =1_G$, the identity in $G$.
\end{enumerate}

\end{enumerate}
\end{thm}
%{\color{blue} I made an $e_G$ above a $1_G$ since we used the latter a few times times below. Feel free to change back to $e_G$ and make other few below that too.}

Outside of a few degenerate cases, see \cite{CARVACHO2021106552}, a tuple $[h; n_1,\ldots, n_r]$  satisfies the necessary arithmetic conditions of Theorem \ref{T:RET} if and only if the numbers $n_1,\dots ,n_r$ are all orders of non-trivial elements of $G$. In particular, for a given $G$, all tuples which satisfy just the arithmetic conditions are easy to compute. We call such tuples {\it potential signatures}. This is in contrast to {\it actual signatures} which are tuples satisfying all conditions of Theorem \ref{T:RET}, and hence are signatures for some $G$ acting on a compact Riemann surface. Since potential signatures are in principle easy to find, and actual signatures not, this leads to the following question: are there groups for which the difference between potential signatures and actual signatures is finite?

Previous work, \cite{Bozlee2014}, showed that asymptotically, as genus increases through the genus spectrum of a group, the proportion of potential signatures to actual signatures goes to $1$, meaning eventually most potential signatures are actual signatures. However, in that same work, the authors exhibited examples of groups for which there are infinitely many potential signatures which are not actual signatures. More recently, in \cite{CARVACHO2021106552}, the authors developed necessary and sufficient conditions for when all but finitely many potential signatures are actual signatures, and it is a corollary of those conditions that all non-abelian finite simple groups have this property.

A natural question, then, is what exactly is the set of  potential signatures which are not actual signatures for particular non-abelian finite simple groups? In the case where the quotient genus is greater than one, the answer is simple: all potential signatures are actual signatures.

\begin{thm} For any non-abelian finite simple group $G$, any potential signature $[h;-]$ (i.e.\ no ramification) or $[h;n_1, \ldots, n_r]$ with $h>1$ and the set of $n_i \in \mathcal{O}(G)$ is an actual signature. \end{thm}

\begin{proof}
Fix a non-abelian finite simple group $G$. We find elements $a_1, b_1, \ldots, a_h, b_h, c_1, \ldots, c_r$ in $G$  so that $\prod_{i=1}^h [a_i,b_i] \prod_{j=1}^{r} c_i= 1_G$, allowing for there to be no $c_i$ in the case where the signature is $[h;-]$.

It is well known that any finite simple group may be generated by two elements in the group, see \cite{MR143801} 
and \cite[Theorem B]{Asch-Gural}.  Choose $a_h$ and $b_h$ to be any two generators of $G$.  Supposing that $r \geq 1$ additionally choose $c_i$ to be arbitrary elements of the group of order $n_i$ for $1 \leq i \leq r$.   Ore's conjecture, proven in \cite{MR2654085}, states that every element of a non-abelian finite simple group is a commutator. In particular $([a_h,b_h]\cdot c_1 \cdots c_r)^{-1}$ is a commutator, call it $[a_1,b_1]$.  If $h>2$ also let $a_i=b_i=1_G$ for $1<i<h$. Then $(a_1, b_1, \ldots, a_h, b_h, c_1, \ldots, c_r)$ is a generating vector for $G$, and thus by Theorem \ref{T:RET}, $G$ acts with signature $[h;n_1, \ldots, n_r]$ on some compact Riemann surface.\end{proof}

The question of which potential signatures are not actual signatures for $h=0$ and $h=1$ is a harder problem, as we no longer have multiple commutators to work with. In this paper we study the question for the alternating group $A_n$ for $n \geq 5$ and $h =1$. Our work relies heavily on the fact that every element of $A_n$ is a commutator (Ore's conjecture in the case of $A_n$), so the techniques we develop cannot be adapted for when $h=0$. We leave that case for future work as it will require a considerably different approach.  

We prove the following theorem. 

\begin{thm}
    For all $n>6$ every potential signature $[1;n_1, \ldots, n_r]$ is an actual signature for $A_n$. In the cases of $n=5$ and $n=6$, the only failures occur for $[1;2]$ and $[1;3]$, respectively.
\end{thm}

The proof of this theorem breaks down into four steps. In Section \ref{S:smallprimes}, we show that provided $2$ or $3$ divides one of the $n_i$, then $[h;n_1,\dots ,n_r]$ is an actual signature for $n\geq 40$. Henceforth, assuming none of the $n_i$ are divisible by $2$ or $3$, in Section \ref{S:twoperiod} we prove the theorem when the action has at least 2 periods for $n\geq 24$, and in Section \ref{S:oneperiod}, when it has 1 period for $n\geq 17$. To finish in Section \ref{S:compute}, we use computational methods to prove the theorem for all $n$ with $5\leq n\leq 40$, which completes the proof. Note by Theorem \ref{T:RET}(1) we cannot have an unramified cover when $h=1$ and $\sigma>1$. We begin in Section \ref{S:background} with  an outline of the proof and statements of core results we use in the proof. Section \ref{S:building} describes how we construct certain elements of $A_n$ which help us in the proof.

\section{Proof Outline and Core Results}\label{S:background}

Throughout the rest of the paper, $\Omega$ will denote the set $\{1, 2 \ldots, n\}$, and henceforth, to avoid confusion, we shall refer to members of the set $\Omega$ as {\it points} and members of the group $A_n$ as {\it elements}. We fix the following further notation and terminology for elements of $A_n$:  

\begin{notation}
\label{not-cycles}
In general, we shall use letters of the Roman alphabet for elements of $A_n$, and letters of the Greek alphabet for individual cycles. Then, for a given $c \in A_n$ with $c =\alpha_1 \dots \alpha_s$ where the $\alpha_i$ are disjoint cycles we define the following:

\begin{itemize}
\item
$\alpha_{i,j}$ is the $j$-th point in the cycle $\alpha_i$ (when the cycle begins with its smallest point).

\item
${\rm Supp} (c)$ denotes the support of $\alpha$ considered as a subset of $\Omega$ and $|{\rm Supp} (c)|$ its size.

\item
${\rm Supp}^c (c)$ denotes the complement of the support of $c$ considered as a subset of $\Omega$ and $|{\rm Supp}^c (c)|$ its size.

\item We say a set of elements in $A_n$ has {\it full support} when the union of the supports of all elements in the set is $\Omega$.

\item  $\text{nc}(c)$ denotes the number of non-trivial cycles in $c$.

\end{itemize}

\end{notation}

\subsection{Overview of Proof}\label{SS:outline}
Each step of the proof follows a similar argument. Since $h=1$, to show that $[1;n_1,\dots ,n_r]$ is an actual signature for $A_n$, by Theorem \ref{T:RET}, we search for elements  $a, b, c_1, \ldots, c_r \in G$ which generate $G$ such that $c_i$ has order $n_i$ and $[a,b] \cdot c_1 \cdots c_r=1_G$.

Most of our work relies on a key theorem in Bertram \cite[Theorem 2]{Bertram}, and see Theorem \ref{T:Bertram} below, which says that any element of $A_n$ can be written as the product of two cycles of length $\ell$ for all $\ell \geq 3n/4$.  As long as $\ell<n-1$ then these two $\ell$-cycles will be conjugate to each other, since elements in $A_n$ with the same cycle structure are all in the same conjugacy class unless the cycle structure (including fixed points) consists of distinct odd numbers. See also \cite{MR3719453} for similar findings on generating $A_n$ with two cycles.

If an element  $g \in A_n$ can be written as a product of two elements of a cycle type such that all elements of that cycle type are in the same conjugacy class, then $g$ is a commutator.  For if $a$ and $a'$ are the two elements in the same conjugacy class, then $a'=b^{-1}ab$ for some $b \in A_n$ and so $g=a^{-1} \cdot a'=[a,b]$.

In particular if we take the element of $A_n$ to be $(c_1 \cdots c_r)^{-1}$, for $c_i$ of the right orders, this product can be written as a commutator via the product of these two conjugate $\ell$-cycles.  This guarantees conditions (2b) and (2c) in Theorem \ref{T:RET}. We summarize.

\begin{prop}\label{p-conjugates}
Suppose that $\ell_1, \ell_2$ are two $\ell$-cycles in $A_n$ with $\ell <n-1$ and $c_1, \dots ,c_r$ elements of orders $n_1,\dots ,n_r$ such that $\ell_1 \ell_2 =(c_1\cdots c_r)^{-1}$ and $A_n=\langle \ell_1, \ell_2, c_1, \dots , c_r\rangle$. Then $[1;n_1,\dots ,n_r]$ is an actual signature for $A_n$.
\end{prop}

In fact, $\ell_1, \ell_2$ need not be cycles in Proposition \ref{p-conjugates}, just conjugate elements of $A_n$, which is something we will use in part of the proof of Theorem \ref{T:1-k}.

What is left to check, then, is that these elements do generate $A_n$, and we do this using two methods.  One method uses a theorem of Miller  \cite[pg.\ 25]{Miller} which guarantees a transitive group is at least $A_n$ as long as it contains a cycle of prime order for certain primes. The other method uses a well-known generalization  of a classical theorem attributed to Jordan  (\cite{MR1503659} and \cite{MR1503635}): a subgroup of $S_n$ which is a primitive permutation group (see the next section for a definition of {\em primitive}) and which contains a cycle with $m$ fixed points for some  $m \geq 3$ must be $A_n$ or $S_n$ (see \cite{Jones} for a proof of the generalization). This reduces the problem to proving that the subgroup we have constructed is primitive (it will have a cycle with at least $3$ fixed points by the way we construct it).  

\subsection{Core Results}\label{SS:priors}

Since our proofs use several theorems from other papers to verify transitivity and primitivity of the groups we generate,  we include the statement of these theorems here, and provide proofs for several other facts we will need later. We begin with the theorems  we use to create the commutator of $a$ and $b$.

\begin{thm}[Bertram, Theorem 2, \cite{Bertram}]\label{T:Bertram} Given an element  $c \in A_n$ and $\ell>0$ an integer satisfying $$\frac{ |\textnormal{Supp}(c )| + \textnormal{nc}(c)}{2} \leq \ell \leq n.$$ Then there exist two $\ell$-cycles whose product is $c$.  \end{thm}

 Bertram additionally proves (\cite[Lemma 3]{Bertram}) that the value $(|\text{Supp}(c)| + \text{nc}(c))/2$ is bounded above by $\lfloor 3n/4 \rfloor$. So for every $\ell \geq \lfloor 3n/4 \rfloor$ and any $c \in A_n$ we can write $c$ as the product of two cycles of order $\ell$. In the latter part of the proof of Theorem \ref{T:1-k}, we are unable to use the cycles given by Bertram, and so we instead use the following.
 
\begin{thm}[Xu, Lemma, pg. 339 \cite{Xu}]\label{T:Xu} For $n$ odd, let $a=(1\ 2)(3\ 4\ \ldots \ n-1)$, and for $n$ even, let $a=(1\ 2)(3 \ 4\ \ldots\  n)$. Then every element of $A_n$ can be expressed as a product of two permutations which are both conjugate to $a$ and whose cycles of length $2$ have at least one point in common.
\end{thm}

In Section \ref{S:oneperiod} we need to add a condition on the cycles in Theorem \ref{T:Xu}: we will require that the cycles of length $2$ share exactly one point. Xu's proof is constructive by induction on $n$, and working through that proof we see that the only permutations which must be written with their $2$-cycles identical are those conjugate to $(1\ 2\ 3)$, which we are able to avoid.  \\

In order to prove that the elements we have found generate $A_n$ we use the following theorems.

\begin{thm}[Miller pg. 25, \cite{Miller}]\label{T:Miller} If a transitive group of degree $n$ contains a cycle of prime order $p$, where $p$ satisfies the condition $n/2 <p \leq n -3$, it must be either alternating or symmetric.\end{thm}

\begin{thm}[Jones,  Corollary 1.3 \cite{Jones}]\label{T:jones} Let $G$ be a primitive permutation group of finite degree $n$, containing a cycle with $m$ fixed points. Then $G\geq A_n$ if $m\geq 3$.
\end{thm}

We remind the reader of the definition of primitive, and set some terminology to be used below.
\begin{defn} A  nonempty subset $B$ of $\Omega$ is called a {\it block} if every element of $G$ either preserves $B$ or sends all points of $B$ to points outside of $B$. The {\it length} of a block is $|B|$. A group $G$ is {\it primitive} if it is transitive and its only blocks are the {\it trivial blocks}, namely those with $|B|=1$ or $B=\Omega$.   \end{defn}

Notice that to prove primitivity, we first need to prove that a group is transitive. We primarily do this using the following lemma and its corollary.

\begin{lem}\label{L:transitive} Let $c_1 =\alpha_1 \dots \alpha_t$ and $c_2 =\beta_1\dots \beta_u$ with $u\geq t$ be two elements of $A_n$ such that 
\begin{enumerate}
    \item 
    for $i\leq t-1$, $\beta_i$ contains a point from $\alpha_i$ and $\alpha_{i+1}$, and 
    \item 
    for $i\geq t$, $\beta_i$ contains some point from ${\rm Supp} (c_1)$.
\end{enumerate}
Then $\langle c_1 ,c_2 \rangle$ is transitive on ${\rm Supp} (c_1) \cup {\rm Supp} (c_2)$.
\end{lem}

\begin{proof}

 %It suffices to show that $\langle c_1 ,c_2 \rangle$ is transitive on ${\rm Supp} (c_1)$ since by construction, every point from ${\rm Supp} (c_2)$ is included in a cycle of $c_2$ with a point from $c_1$.  We proceed by induction on $t$. 

%We first prove that $\langle c_1, c_2\rangle$ is transitive on ${\rm Supp}(c_1)$ by induction on $t$. If $t=1$, the conclusion is trivial, so assume $\langle c_1, c_2\rangle$ is transitive on  ${\rm Supp}(c_1)\cup {\rm Supp}(c_2)$ when $t=l$ and prove it for $t=l+1$. By induction, we know that $\langle c_1, c_2\rangle$ is transitive on  ${\rm Supp}(\alpha_1\cdots \alpha_l)$. To show that any orbit also contains the points in  $\alpha_{l+1}$, take a point $f\in \alpha_l$ and apply an appropriate power of $c_1$ to map  $f$ to a common point $g$ in $\beta_l$. Then apply an appropriate power of $c_2$ to map $g$ to a  common point in $\alpha_{l}$.

We first show by induction on $t$ that all points of ${\rm Supp}(c_1)$ lie in a single $\langle c_1,c_2\rangle$--orbit. The case $t=1$ is immediate. Assume the claim holds for $t=l$; that is, all points of ${\rm Supp}(\alpha_1\cdots\alpha_l)$ lie in one orbit of $\langle c_1,c_2\rangle$. To extend this to $t=l+1$, take a point $f\in \alpha_l$ and apply an appropriate power of $c_1$ to map  $f$ to a common point $g$ in $\beta_l$. Then apply an appropriate power of $c_2$ to map $g$ to a  common point in $\alpha_{l}$.

Now, any point of $c_2$ will be in the same orbit as any point of $c_1$ since, by construction, every point from ${\rm Supp} (c_2)$ is included in a cycle of $c_2$ with some point from $c_1$. Thus $\langle c_1 ,c_2 \rangle$ is transitive on ${\rm Supp} (c_1) \cup {\rm Supp} (c_2)$. \end{proof}

\begin{cor}\label{C:transitive} If two $\ell$-cycles  have full support between them, and they share at least one point, then the subgroup they generate is transitive.
\end{cor}

And finally, in cases where we do not use Miller's theorem directly, we need to prove primitivity using a theorem inspired by Miller's work.

\begin{thm}\label{T:primitive} If a transitive permutation group $G$ of degree $n$ contains an $\ell$-cycle so that $\gcd(n,\ell)=1$ and $\ell > \frac{n}{2}$, then the group must be primitive.\end{thm}

In order to prove Theorem  \ref{T:primitive} we first cite several known facts about primitive permutation groups which we use in the proof, see \cite{Wielandt} for more details. 

If $B$ is a block of $G$, then so is $B^g = \{b^g : b \in B\}$ for every $g \in G$, and the set of all such $B^g$ is called a \emph{complete block system} for $G$.
When $G$ is transitive, all members of a complete block system have the same size,
as they are permuted transitively by $G$, and so their common size divides the
degree of $G$ (as a permutation group).

%If $B$ is a block of $G$, then $B^g=\{g(b) : b \in B\}$ is also a block. We call $B^g$ a {\em conjugate} of $B$.  The set of all conjugate blocks of a given block is called a {\it complete block system} and all blocks of a complete block system are the same size. In the case of transitive groups, a complete block system covers all points in $\Omega$ and so:

%\begin{prop}[\cite{Wielandt}, Proposition 6.3]\label{Pr:subgpgs} The length of a block of a transitive group $G$ divides the degree of $G$ as a permutation group. \end{prop}

%Additionally if $H \leq G$ then every block of $G$ is also a block of $H$. 

\begin{proof}[Proof of Theorem 2.8]
Let $c$ denote the $\ell$-cycle in $G$. Consider any complete block system for $G$, with $n/k$ blocks of equal size $k$.
Now if every block contains either  
(i) only points moved by $c$, or  
(ii) only points fixed by $c$,  
then the union of the blocks of type (i) has size $\ell$, while the union of the blocks of 
type (ii) has size $n-\ell$. But then $k$ divides both $\ell$ and $n-\ell$, which is 
impossible since $\gcd(\ell, n-\ell)=\gcd(\ell,n)=1$.

Hence some block $B$ contains both a point $f$ moved by $c$ and a point $g$ 
fixed by $c$. But then $B$ is preserved by $c$, since it contains $g$, and so 
$B$ contains all $\ell$ points of the $\langle c\rangle$-orbit of $f$.
This implies that $|B|\ge \ell > n/2$, and so $|B|=n$. Hence $G$ is primitive.

%(Theorem \ref{T:primitive}) Suppose, by contradiction, that there is a nontrivial block $B \subsetneq \Omega=\{1, \ldots, n\}$. Let $c$ denote the $\ell$-cycle in $G$. Since $G$ is transitive, there is an induced transitive action on the conjugate blocks. As we noted above, each of those conjugate blocks has the same size as $B$ which divides $n$ and is relatively prime to $\ell$.  Then each conjugate block will contain a fixed point of $c$, else we would have a conjugate block only containing points moved by $c$ which would be a block of $\langle c \rangle \leq G$.  This block would then have order dividing $n$ and $\ell$, a contradiction (by Proposition \ref{Pr:subgpgs} applied to $G$ and $\langle c \rangle$).  

%Consider a conjugate block which also contains a point moved by $c$. Call the block $B'$ and the point $j$.  Since $B'$ contains a point fixed by $c$, we know $(B')^c=B'$ and so $c(j) \in B'$.  If we iteratively apply $c$, we see that $c^i(j)$ is also in $B'$ for each $1\leq i \leq \ell$ and so  every point moved by $c$ must also be in $B'$.  But then $|B'|\geq \ell >n/2$, a contradiction. 

\end{proof}

\section{Building Elements in $A_n$}\label{S:building}

Much of our work will depend on being able to (1) create elements with large support in $A_n$ and (2) generate transitive subgroups of $A_n$. Therefore, before the main proof, we focus on how to construct elements which satisfy these conditions.

%Much of our work will depend on two conditions:  being able to create elements with large support in $A_n$, and being able to generate transitive subgroups of $A_n$. Therefore, before the main proof, we focus on these ideas first.  

\begin{lem}\label{lem-support}
Let $n_1$ be the order of an element of $A_n$ for $n\geq 12$. Then unless $n_1$ is prime and greater than $\lfloor \frac{n}{2} \rfloor$, we can find an element $c \in A_n$ of order $n_1$ so that $|{\rm Supp} (c)| >  \frac{2n}{3}$.
\end{lem}

\begin{proof} We construct an element of order $n_1$ with the appropriate sized support as follows. Given the prime factorization $p_1^{a_1}p_2^{a_2} \cdots p_m^{a_m}$ of $n_1$, with $p_i<p_{i+1}$, we first construct an element with one cycle of order $p^{a_i}$ for each $i$, adding an additional 2-cycle if $p_1=2$. We then fill in additional cycles of lengths dividing $n_1$ (coupled with additional $2$-cycles when adding cycles of even length) until $|{\rm Supp}^c (c)| <p_1$ (or $|{\rm Supp}^c (c)| <n-4$ if $p_1=2$). Since $c$ is only made up of cycles whose lengths divide $n_1$, including at least one cycle of length $p_i^{a_i}$ for each $1 \leq i \leq m$, we see that $c$ has order $n_1$. Also note that if $n_1=p_1$ is prime, then $c$ has at least two $p_1$-cycles since $n_1\leq \lfloor \frac{n}{2} \rfloor$. In particular, by construction, $c$ will always contain at least two disjoint cycles.

We now show $|{\rm Supp} (c)| >  \frac{2n}{3}$. We know that each of the cycles in $c$ has length bounded below by $p_1$ so, since $c$ contains at least two cycles,  we know that $|{\rm Supp} (c) |\geq 2p_1$. If it were true that $|{\rm Supp} (c) |\leq  \frac{2n}{3} $ then $2p_1\leq \frac{2n}{3}$ and so $p_1\leq n/3$. However, this violates the construction of $c$. Specifically, if $p_1$ is odd, we could append an extra $p_1$-cycle, contradicting the assumption that $|{\rm Supp}^c (c) |<p_1$, and if $p_1=2$,  then since we are assuming $n\geq 12$, the number of points not in the support of $c$ is at least $n/3 \geq 12/3=4$, contradicting the fact that $|{\rm Supp}^c (c)| <n-4$. \end{proof}

The following lemma, and in particular the construction of specific elements of $A_n$ in the proof, are essential for proving the existence of generating vectors for signatures with arbitrarily long periods. They are also used heavily in our computations for $n<40$.

\begin{lem}

\label{lem-trans} Given any two elements $b_1, b_2 \in A_n$, there exist elements $c_1, c_2 \in A_n$ with the same cycle types as $b_1$ and $b_2$, respectively, so that $\langle c_1,c_2 \rangle$ is transitive on ${\rm Supp} (c_1)\cup {\rm Supp} (c_2)$. 

%The points of any $c_1,c_2\in A_n$ can always be chosen so that $\langle c_1,c_2 \rangle$ is transitive on ${\rm Supp} (c_1)\cup {\rm Supp} (c_2)$.
\end{lem}

\begin{proof}
Let $b_1 = \alpha_1 \dots \alpha_t$ and $b_2 = \beta_1 \dots \beta_u$, and, without loss of generality, suppose that $u \ge t$. We let $c_1 = b_1$ and let $c_2 = \delta_1 \delta_2 \dots \delta_u$, with $u$ disjoint cycles $\delta_i$, all having the same length as $\beta_i$, be defined as follows:

%define the cycles of $c_2 = \delta_1 \dots \delta_u$, where each $\delta_i$ has the same

$$\delta_1=(\alpha_{1,1} \ \alpha_{2,1} \ \gamma_{1,1}\ \dots \ \gamma_{1,s_1} )$$
$$\delta_2=(\alpha_{2,2} \  \alpha_{3,1} \ \gamma_{2,1}\ \dots \ \gamma_{2,s_2}  )$$
$$\delta_3= (\alpha_{3,2} \  \alpha_{4,1} \ \gamma_{3,1}\ \dots \  \gamma_{3,s_3}  )$$  
$$\vdots$$
$$\delta_{t-1}= (\alpha_{t-1,2}\   \alpha_{t,1}\  \gamma_{t-1,1}\ \dots \ \gamma_{t-1,s_{t-1}} )$$
$$\delta_{t} =(\alpha_{\beta_t}\ \gamma_{t,1}\ \dots \ \gamma_{t,s_t} )$$
$$\vdots$$
$$\delta_{u} =(\alpha_{\beta_u}\  \gamma_{u,1} \ \dots \ \gamma_{u,s_u} )$$
where $\alpha_{\beta_t},\dots, \alpha_{\beta_u}\in {\rm Supp} (c_1)$ but none of which have already appeared in a cycle of $c_2$, and the $\gamma_{i,j}$ are any points chosen from $\Omega$ so that the cycles of $c_2$ remain disjoint. (We will impose more restrictions on the $\gamma_{i,j}$ when we use this construction in Section \ref{S:twoperiod}.) By construction, $c_1$ and $c_2$ satisfy the conditions of Lemma \ref{L:transitive}, and hence $\langle c_1, c_2\rangle$ is transitive on ${\rm Supp} (c_1)\cup {\rm Supp} (c_2)$.

%Let  $\alpha_1 \dots \alpha_t$ represent the cycle type of $b_1$ and  $\beta_1, \dots , \beta_u$ represent the cycle type of $b_2$.  Without loss of generality, suppose that $\text{nc}(b_2) \geq \text{nc}(b_1)$. Choose the points of $c_1$ arbitrarily. Then, we define the cycles of $c_2$ as follows:
%$$\beta_1=(\alpha_{1,1} \ \alpha_{2,1} \ \gamma_{1,1}\ \dots \ \gamma_{1,s_1} )$$
%$$\beta_2=(\alpha_{2,2} \  \alpha_{3,1} \ \gamma_{2,1}\ \dots \ \gamma_{2,s_2}  )$$
%$$\beta_3= (\alpha_{3,2} \  \alpha_{4,1} \ \gamma_{3,1}\ \dots \  \gamma_{3,s_3}  )$$  %{\color{red} changed above from $\gamma_{4,s_3}$ to $\gamma_{3,s_3}$ confirming that is correct?}
%$$\vdots$$
%$$\beta_{t-1}= (\alpha_{t-1,2}\   \alpha_{t,1}\  \gamma_{t-1,1}\ \dots \ \gamma_{t-1,s_{t-1}} )$$
%$$\beta_{t} =(\alpha_{\beta_t}\ \gamma_{t,1}\ \dots \ \gamma_{t,s_t} )$$
%$$\vdots$$
%$$\beta_{u} =(\alpha_{\beta_u}\  \gamma_{u,1} \ \dots \ \gamma_{u,s_u} )$$
%where $\alpha_{\beta_t},\dots, \alpha_{\beta_u}\in {\rm Supp} (c_1)$ but none of which have already appeared in a cycle of $c_2$, and the $\gamma_{i,j}$ are any points chosen from $\Omega$ so that the cycles of $c_2$ remain disjoint. (We will impose more restrictions on the $\gamma_{i,j}$ when we use this construction in Section \ref{S:twoperiod}.) By construction, $c_1$ and $c_2$ satisfy the conditions of Lemma \ref{L:transitive}, and hence $\langle c_1, c_2\rangle$ is transitive on ${\rm Supp} (c_1)\cup {\rm Supp} (c_2)$. 

\end{proof}

\section{The case when periods are divisible by small primes}\label{S:smallprimes}

We run into difficulties when there is a period that is divisible by either $2$ or $3$, so we deal with those cases first. We will need the following proposition.

\begin{prop}
\label{prop-prime}
For any integer $n\geq 40$, there is always a prime between $\lfloor \frac{3n}{4}\rfloor +3$ and $n-3$.  
\end{prop}

\begin{proof}
In \cite{Nagura} it is shown that for $\displaystyle m\geq 25$ there is always a prime between $m$ and $\frac{6m}{5}$.  Letting $m=\lfloor \frac{3n}{4}\rfloor +3$ we can determine a similar bound in terms of $n$. Specifically, when $m\geq 25$ we have $n\geq 30$. Since there is always a prime between $m$ and $6m/5$ when $m\geq 25$ it follows that there is always a prime between $m=\lfloor \frac{3n}{4}\rfloor +3$ and $$\frac{6m}{5}=\frac{6}{5}\cdot \left(\left\lfloor \frac{3n}{4}\right\rfloor +3 \right)=\left\lfloor \frac{18n}{20}\right\rfloor +\frac{18}{5}$$ when $n\geq 30$. Observe that $$n-3\geq  \left\lfloor \frac{18n}{20}\right\rfloor +\frac{18}{5}$$ for $n\geq 61$, so it follows that there is always a prime between $\lfloor \frac{3n}{4}\rfloor +3$ and $n-3$ when $n\geq 61$. For values between $40$ and $60$ we can check directly that there is a prime between $\lfloor \frac{3n}{4}\rfloor+3$ and $n-3$.   
\end{proof}

\begin{thm}
\label{T:smallprimes}
Suppose that $[1;n_1,\dots ,n_r]$ is a potential signature for $A_n$ for $n\geq 40$ where one of the $n_i$ is divisible by either $2$ or $3$. Then $[1;n_1,\dots ,n_r]$ is an actual signature for $A_n$.
\end{thm}

\begin{proof} Using Proposition \ref{p-conjugates}, it suffices to find two $\ell$-cycles $\ell_1, \ell_2 \in A_n$ with $\ell <n-1$, and elements $c_1, \dots ,c_r$ of orders $n_1,\dots ,n_r$ such that $\ell_1 \ell_2 =(c_1\cdots c_r)^{-1}$ and $A_n=\langle \ell_1, \ell_2, c_1, \dots , c_r\rangle$.

First suppose that $r=1$, so the signature is of the form $[1;k]$. By assumption, $k$ is divisible by $2$ or $3$, and so, using the construction from the proof of Lemma \ref{lem-support}, we can build an element $c_1\in A_n$ of order $k$ with ${\rm Supp} (c_1) \geq n-3$. Using Theorem \ref{T:Bertram}, there are two conjugate cycles $\ell_1$ and $\ell_2$ of length $\lfloor3n/4\rfloor$ whose product is $c_1^{-1}$. Note that the combined size of the supports of $\ell_1$ and $\ell_2$ is at least $n-3$ as the product must stabilize all the points in $c_1$.  By Proposition \ref{prop-prime} there is a prime $p$ between $\lfloor 3n/4\rfloor +3$ and  $n-3$. Now, using the construction in \cite{Bertram}, we can add additional points to each of the cycles $\ell_1$ and $\ell_2$ to create cycles of length $p$ whose combined support has size $n$. Specifically, we add points to the cycles $\ell_1$ and $\ell_2$ by first adding the fixed points of $c_1$ and then additional points as needed, observing that since $p\geq \lfloor 3n/4 \rfloor +3$, at least three points will always be added.

%I think we have the wrong thing in the last paragraph. It should be "In either case, the group $\langle l_1,l_2,c_1,c_2\rangle$ is clearly transitive...". For the earlier comments, I think maybe we just remind the reader that we are building a generating vector here. Also, we should probably remark that we can, without loss of generality, assume that either $n_1$ or $n_2$ is divisible by $2$ or $3$ (we do need this as we are only assuming a single one is)

Now suppose that $r\geq 2$ and, without loss of generality, assume that $n_1$ is the period divisible by $2$ or $3$. Using the construction given in Lemma \ref{lem-support}, we may assume the cycle type of $c_1$ is such that ${\rm Supp} (c_1) \geq n-3$. Next, using Lemma \ref{lem-trans}, we may construct $c_1$ and $c_2$ so that $\langle c_1 , c_2\rangle$ acts transitively on ${\rm Supp} (c_1)\cup {\rm Supp} (c_2)$, a set of size at least $n-3$. We then choose arbitrary elements of orders $n_3,\dots ,n_r$ for $c_3,\dots ,c_r$, and we imitate the construction of $\ell_1$ and $\ell_2$ as for the case $[1;k]$ above. Specifically, start with cycles $\ell_1$ and $\ell_2$ of lengths $\lfloor 3n/4 \rfloor$ whose product is $(c_1\cdots c_r)^{-1}$ and then add points to ensure they have prime length and full support between them.

In either case, the group  $\langle \ell_1 ,\ell_2, c_1, c_2, 
\ldots, c_r\rangle$,  or $\langle \ell_1 ,\ell_2, c\rangle$ for $r=1$, is clearly transitive by Lemma \ref{L:transitive}, and since $p>n/2$ and $p<n-3$, by Theorem \ref{T:Miller} they in fact generate all of $A_n$. Thus by Proposition \ref{p-conjugates}, $[1;n_1,\dots ,n_r]$ is an actual signature. \end{proof}

When $r=1$, the previous proof relied only on the size of the support of $c$ being at least $n-3$. Later, in the one period case, we will use that corollary more generally, so we state it here.

\begin{cor}
\label{C:largesupport}
For $n \geq 40$, if $A_n$ contains an element of order $k$ whose support has size at least $n-3$, then $[1;k]$ is an actual signatures for $A_n$.
\end{cor}

\section{The case $[1;n_1, n_2, \ldots, n_r]$ for $r\geq 2$}\label{S:twoperiod}

In this section, we shall show that if $r\geq 2$, then a potential signature $[1;n_1, n_2, \ldots, n_r]$ for $A_n$ is always an actual signature for sufficiently large $n$ provided none of the $n_i$ are divisible by $2$ or $3$. We use the following proposition to reduce the problem to proving a certain subgroup is transitive. 

\begin{prop}
\label{prop:trans-support}
Suppose that $[1;n_1, n_2, \ldots, n_r]$ is a potential signature for $A_n$ with $n\geq 24$ and $r \geq 2$. If there exist elements $c_1,\dots, c_r \in A_n$ of orders   $n_1, n_2, \ldots, n_r$ respectively such that
$\langle c_1, \dots, c_r\rangle$ is a transitive subgroup of $A_n$, then $[1;n_1, n_2, \ldots, n_r]$  is an actual signature for $A_n$.
\end{prop}

\begin{proof}
Imitating the proof of Proposition \ref{prop-prime}, we can show that for $n\geq 24$, there is always a prime $p$ with $\lfloor 3n/4 \rfloor \leq p\leq n-3$. Henceforth assume that $n\geq 24$ and let $p$ be a prime with $\lfloor 3n/4 \rfloor \leq p\leq n-3$. 

Fix elements $c_1,\dots, c_r \in A_n$ of orders $n_1, n_2, \ldots, n_r$ respectively such that $\langle c_1, \dots, c_r\rangle$ is transitive. By Theorem \ref{T:Bertram}, since $p\geq \lfloor 3n/4 \rfloor$, we can find two $p$-cycles $\ell_1$ and $\ell_2$ whose product is $(c_1\cdots c_r)^{-1}$. Now, the group $\langle \ell_1 , \ell_2, c_1, \dots ,c_r\rangle$ is transitive and contains a $p$-cycle for $p>\lfloor n/2\rfloor$, so by Theorem \ref{T:Miller} we have $\langle \ell_1 , \ell_2, c_1, \dots, c_r\rangle=A_n$. Thus by Proposition \ref{p-conjugates}, it follows that $[1;n_1, n_2, \ldots, n_r]$ is an actual signature for $A_n$.\end{proof}

We can now prove the main theorem for this section. The proof relies heavily on the construction of elements given in Section \ref{S:building}.

\begin{thm}
\label{thm-twoperiods}
When $n\geq 24$ and none of the $n_i$ are divisible by $2$ or $3$, then a potential signature $[1;n_1,\dots ,n_r]$ for $A_n$ with $r\geq 2$ is always an actual signature.
\end{thm}

\begin{proof}
By Proposition \ref{prop:trans-support}   it suffices to find $c_1,\dots, c_r \in A_n$ of orders   $n_1, n_2, \ldots, n_r$ respectively such that $\langle c_1, \dots, c_r\rangle$ is a transitive subgroup of $A_n$. We do this by showing that when $n \geq 24$ we can choose $c_1$ of order $n_1$ and $c_2$ of order $n_2$ such that the group $\langle c_1,c_2\rangle$ is transitive, thus allowing us to choose arbitrary elements $c_3,\dots ,c_r$ of orders $n_3,\dots ,n_r$ respectively.

We use the construction of $c_1$ and $c_2$ given in the proof of Lemma \ref{lem-trans}. Specifically, we construct $c_2$ as given in that proof, but we choose the $\gamma_{i,j}$ so they are first chosen from ${\rm Supp}^c (c_1)$, with any remaining $\gamma_{i,j}$ coming from ${\rm Supp} (c_1)$ but not having already appeared in a previous cycle. We call the points which are subscripted $\alpha$'s {\em forced points} of $c_2$ and the remaining points which are subscripted $\gamma$'s {\em free points} of $c_2$. For conciseness, we denote the number of free points of $c_2$ by $\mathfrak{F}$. We also note by Lemma \ref{lem-support}, except when $n_1$ is prime and larger than $n/2$, we can assume the size of the support of both $c_1$ and $c_2$ is at least $2n/3$.

We make some initial observations. By construction, $c_2$ is written in disjoint cycle notation, and since we are assuming $u\geq t$ and each cycle contains at least two points, the forced points defined in the construction of $c_2$ all appear. In addition, there are precisely $2(t-1)+(u-(t-1))=t+u-1$ forced points in $c_2$.   Also note that by construction, $c_1$ and $c_2$ satisfy the conditions of Lemma \ref{L:transitive} and hence by that lemma, the group they generate is transitive on ${\rm Supp} (c_1) \cup {\rm Supp} (c_2)$.  Thus to prove the theorem, we just need to show that ${\rm Supp} (c_1) \cup {\rm Supp} (c_2)=\Omega$.  

First consider the case where $u=t=1$. In this case, $c_1$ and $c_2$ are both single cycles, each of prime length strictly bigger than $\lfloor \frac{n}{2} \rfloor$, sharing at least one point. Moreover, due to their length and the way we have constructed them, their combined support must be all of $\Omega$. Hence, by Corollary \ref{C:transitive}, the group they generate is transitive.

To show ${\rm Supp} (c_1) \cup {\rm Supp} (c_2)=\Omega$ for the remaining cases, it suffices to show that $\mathfrak{F}\geq |{\rm Supp}^c (c_1)|$ since in the construction of the cycles of $c_2$, points in ${\rm Supp}^c (c_1)$ are placed first, so this condition guarantees that all points in ${\rm Supp}^c (c_1)$ will appear in $c_2$ before we start placing points from ${\rm Supp} (c_1)$.

We next consider the case where $t=1$ and $u\geq 2$. Since $c_1$ consists of a single cycle of prime length bigger than $\lfloor \frac{n}{2} \rfloor$, it suffices to show that $\mathfrak{F} \geq \lfloor \frac{n}{2} \rfloor$. Now, by assumption, all cycles of $c_2$ have length at least $5$, so at most $1/5$ of the points of $c_2$ are forced, while at least $4/5$ of them are free. Therefore, since $|{\rm Supp} (c_2 )|\geq2n/3$ we have $$\mathfrak{F} \geq \frac{2n}{3}\cdot \frac{4}{5}=\frac{8n}{15} >\bigg\lfloor \frac{n}{2} \bigg\rfloor.$$

Now suppose $t>1$, so $c_1$ has more than one cycle. Note that since $u\geq t$, we also have $u>1$ and so the size of the support of each of $c_1$ and $c_2$ is at least $2n/3$. Thus we just need to show that $\mathfrak{F} \geq n/3$. Once again, all cycles of $c_2$ have length at least $5$, and by construction, each cycle has at most two forced points. Thus at least $3/5$ of the points of $c_2$ will be free, giving $$\mathfrak{F}\geq \frac{2n}{3} \frac{3}{5} =\frac{2n}{5}>\frac{n}{3}.$$
\end{proof}

\section{The case $[1;k]$}\label{S:oneperiod}

For  $n>4$ the signature $[1;k]$ is a potential signature for all $k \in \mathcal{O}(G)$.  As noted above, for small $n$ there are two cases where potential signatures of this form are not actual signatures. When $n=5$ the potential signature $[1;2]$ is not an actual signature, and when $n=6$ the potential signature $[1;3]$ is not an actual signature. However, for larger $n$, all potential signatures of this form are actual  signatures.

\begin{thm}\label{T:1-k} For all $n>6$, each potential signature $[1;k]$ is an actual signature for $A_n$.  \end{thm}

\begin{proof}

In Section \ref{S:compute} we generate specific examples for each $k$ with $6<n<40$ using computer programs. Thus we only need to prove the theorem for $n \geq 40$. By Theorem \ref{T:smallprimes} we assume $2 \nmid k$ and $3 \nmid k$ (this restriction will be used later in the proof). For $n$ odd and $n$ even, we need  different cycle structures so we split the proof into two parts.  

First assume $n$ is odd. Fix an element $c \in A_n$ of order $k$. Apply Bertram's constructions to write $c$ as the product of two $\ell$-cycles of lengths $n-4$. (Note that $n-4>3n/4$ whenever $n>16$.)

Transitivity follows from Lemma \ref{L:transitive} if we can show that the two $\ell$-cycles have full support between them.  To guarantee full support, we study the construction of the $\ell$-cycles from \cite{Bertram}.  Bertram's work guarantees $\ell$-cycles of length $\lfloor \frac{3n}{4} \rfloor$, and those $\ell$-cycles can be extended by adding the points of $\text{Supp}^c(c)$. Once we have done so, we guarantee full support between the two $\ell$-cycles. And so we just need to confirm that we can add all the points of the complement before $\ell$ is larger than  $n-4$.  By Lemma \ref{lem-support}, we can guarantee $\text{Supp}(c) \geq \frac{2n}{3}$.  We also know that $nc(c)$ (the number of nontrivial cycles) is at most $n/5$ (since $2 \nmid k$ and $3\nmid k$).  Putting that together we need to ensure

$$\frac{\text{Supp}(c)+\text{nc}(c)}{2} +n - \text{Supp}(c) \leq n-4$$
or equivalently
$$\frac{\text{nc}(c)}{2} + 4 \leq \frac{\text{Supp}(c)}{2}.$$
Filling in the bounds for $\text{Supp}(c)$ and $\text{nc}(c)$, we see that the inequality will be true if $n \geq 17$.

If, however, Lemma \ref{lem-support} fails then $c$ is a $p$-cycle with $p > \lfloor \frac{n}{2} \rfloor$. This means $\text{nc}(c)=1$ and $\text{Supp}(c)=p$.  Then $\frac{\text{nc}(c)}{2} + 4 \leq \frac{\text{Supp}(c)}{2}$ will be true when $p \geq 9$, or when $n > 14$. 

Finally, to prove primitivity, observe that since $n$ is odd, the $\gcd(n,n-4)=1$ and because of the way we created these $\ell$-cycles, primitivity will follow from Theorem \ref{T:primitive}.

For $n$ even, a similar argument to the one we made in the odd case could fail to give primitivity if $3$ divides $n$. Instead, write the prime factorization $k=p_1^{a_1}\cdots p_m^{a_m}$ for distinct primes $p_1,\dots, p_m$ and $p_i>3$, and choose $c$ to be an element of order $k$ with only one cycle of length $p_i^{a_i}$ for each $i$. 

By Theorem \ref{T:Xu}, this element can be written as a product of two elements conjugate to $(1\  2)(3 \ \dots  n)$ and then since $3$ does not divide $k$, the two involutions will share precisely one point. Hence the  group generated by $c$ and the two conjugate elements will clearly be transitive. If $m\geq 2$, then this group will also be primitive since it will contain two cycles of distinct prime power lengths (take appropriate powers of $c$). The size of any nontrivial block will have to divide both of these values, so will have to be $1$.

This leaves the case $m=1$ with $k=p_1^{a_1}$. If $p_1$ is relatively prime to $n$, then we get primitivity in almost all cases since any block size will have to divide $n$ and $p_1^{a_1}$. There is one small issue in this case with applying Theorem \ref{T:jones}.  When $k=p_1^{a_1}$ and $n-k< 3$ we have less than $3$ fixed points, but in this case we use Corollary \ref{C:largesupport} to show the signature $[1;k]$ is an actual signature for $n\geq 40$. 

Otherwise $p_1$ divides $n$, and we can instead throw away our original choice for $c$ and replace it with an element of full support, in which case Corollary \ref{C:largesupport} again shows the signature $[1;p_1^{a_1}]$ is an actual signature for sufficiently large $n$. \end{proof}

\section{Computations for small $n$}\label{S:compute}

For $n<40$ we use computer programs to help prove that every potential signature is an actual signature.  Of course, for each $n$, there are an infinite number of potential signatures so to successfully check small $n$ values, we need  to reduce that list to a manageable finite number.  

We first find generating vectors for  $[1;k]$ for each $k$ in $\mathcal{O}_G$ (this will complete the proof of Theorem \ref{T:1-k}). Suppose we have a generating vector $(a,b,c)$ corresponding to $[1;k]$. If $k$ is odd, this generating vector guarantees a generating vector for the signature $[1;k,k]$: simply take $(a,b,c^2, c^{-1})$, since $c^2$ and $c^{-1}$ will have the same order as $c$. Note we can repeat this idea to create a generating vector for any signature $[1;k, \ldots,k]$.

 If $k$ is even, the existence of a generating vector for $[1;k]$ guarantees a generating vector for $[1;k,k,k]$: take $(a,b,c, c^{-1},c)$ as a generating vector (and similarly any odd number of $k$'s).  For $k$ even, then, we must also search for generating vectors for $[1;k,k]$.  Once we have those generating vectors we can find a generating vector for $[1;k,k,k,k]$ (and indeed any even number of $k$'s) using the same idea as we did for an odd number of $k$'s.
 
 Finally, for every pair $n_1,n_2 \in \mathcal{O}_G$ with $n_1 \neq n_2$ we demonstrate elements in $A_n$ of order $n_1$ and $n_2$  which generate the group $A_n$.  This guarantees that for any signature $[1;n_1, \ldots, n_r]$ in which at least two distinct values appear among the $n_i$ we can find a generating vector.  We simply pick two $n_i$ that are different and use our generated example of pairs which generate the group. Additionally we choose any elements in $A_n$ with orders equal to the remaining $n_i$.  Then we use Proposition \ref{prop:trans-support} in much the same way as we did in the second to last paragraph of Theorem \ref{thm-twoperiods}.

The computer code (written in Magma \cite{magma} and Python), its output (in particular generating vectors for each of the signature families described above) and more details may be found at \cite{git_repo}.  This work utilized resources from Unity, a collaborative, multi-institution high-performance computing cluster managed by UMass Amherst Research Computing and Data.

%\section{Acknowledgement}
%The authors would like to thank the reviewers for their careful read of the paper and their suggestions for improvement of the exposition and presentation of the material. 

\bibliographystyle{abbrv}
\bibliography{anbib}

\end{document}